\documentclass[english,12pt]{amsart}
\usepackage{amsmath}
\usepackage{amssymb}
\usepackage{amsthm}
\usepackage[T1]{fontenc}
\usepackage{color}
\usepackage{textcomp}
\usepackage[]{enumerate}
\usepackage{pst-node}
\usepackage{tikz-cd}
\usepackage[english]{babel}

\newtheorem{theorem}{Theorem}
\newtheorem{lemma}[theorem]{Lemma}

\newtheorem{corollary}[theorem]{Corollary}

\newcommand{\R}{\mathbb{R}}
\newcommand{\N}{\mathbb{N}}
\newcommand{\e}{\varepsilon}
\newcommand{\To}{\rightarrow}
\newcommand{\dP}{\partial P}
\newcommand{\V}{\mathcal{V}}
\newcommand{\Z}{\mathbb{Z}}
\def\U{\mathcal{U}}
\def\S{\mathbb{S}^1}
\def\A{\bar{A}}
\def\B{\bar{B}}
\def\C{\bar{C}}
\def\D{\bar{D}}

\begin{document}
\title[Topologically weakly mixing  polygonal billiards]{Topologically weakly mixing   polygonal billiards}

\author{Jozef Bobok}

\author{Serge Troubetzkoy}

\address{Department of Mathematics of FCE\\Czech Technical University in Prague\\
Th\'akurova 7, 166 29 Prague 6, Czech Republic}
\email{jozef.bobok@cvut.cz}

\address{Aix Marseille Univ, CNRS, Centrale Marseille, I2M, Marseille, France}
\address{postal address: I2M, Luminy, Case 907, F-13288 Marseille Cedex 9, France}
\email{serge.troubetzkoy@univ-amu.fr}
\urladdr{http://iml.univ-mrs.fr/{\lower.7ex\hbox{\~{}}}troubetz/} \date{}

\thanks{We thank the A*MIDEX project (ANR-11-IDEX-0001-02), funded itself by the ``Investissements d'avenir'' program of the French Government, managed by the French National Research Agency (ANR).  We gratefully acknowledges the support of project  APEX ``Systèmes dynamiques: Probabilités et Approximation Diophantienne PAD'' funded by the Région PACA}

\begin{abstract}
We show that in a typical polygon the billiard map as well as its associated  subshift obtained by coding orbits by the sequence of sides
they visit are topologically weakly mixing.
\end{abstract}\maketitle

\section{Introduction}\label{s1}
A  point mass, moves inside a planar polygon
with unit speed along a straight line until it reaches the boundary, then instantaneously changes direction according to the mirror
law:  ``the angle of incidence is equal to the angle of reflection,'' and
continues along the new line. The first mathematical article about the billiard in a polygon dates from 1905 \cite{Le}.
Billiards in polygons come in two different types.  If the polygon is rational, i.e., all angles between sides are rational
multiple of $\pi$, the dynamics differs from the irrational case.
If one fixes the initial direction of the
billiard in a rational polygon, then  the directions taken by billiard orbit is a finite set, this plays the role of an integral of the motion; the invariant set
of points in the phase space taking these directions is in fact  a translation surface. In the rational polygon case much is known about the dynamics of the billiard flow restricted to
the associated  translation surface for a typical initial direction (i.e., a typical value of the integral) as well as about  the restriction of the  billiard map (the
first return map of the billiard flow to the boundary of the polygon) to the associated translation surface; we will refer to these restrictions as the directional billiard flow and the directional billiard map.
For all but countably many values of this integral the directional billiard map  is minimal \cite{Ke} (and thus the directional billiard flow
is minimal as well).
Kerchoff, Masur and Smillie showed that for almost every value of the integral the directional billiard flow is ergodic with respect to the invariant area measure, in fact uniquely ergodic \cite{KeMaSm} (and thus the directional billiard map is uniquely ergodic as well).
Katok showed that the
directional billiard map and flow are not mixing for any value of the integral \cite{Ka}, while Chaika and Fickenscher gave
examples of rational polygons for which the directional billiard map is topologically mixing for certain specific values of the integral \cite{ChFi}.
\'Avila and Forni has shown that generic interval exchange transformations are weakly mixing  and the straignt line flow generic translation surfaces is weakly mixing in almost every direction \cite{AvFo}. Unfortunately one can not deduce weakly mixing results for  for rational
polygonal billiards from this result since the set of translation surfaces arising from rational polygonal billiards is of measure 0
in the set of all translation surface.  None the less we do
have examples of rational polygons for which for some of the directional flows are weakly mixing. \'Avila and Delecroix
showed that for
almost every direction the directional billiard flow is weakly mixing in regular polygons  other than the lattice ones \cite{AvDe}
and M\'alaga Sabogal and Troubetzkoy showed the analogous  result  for typical polygons with vertical and horizontal sides \cite{MSTr}.
Boshernitzan et.\ al.\ have shown that in a rational polygon the set of periodic orbits of the billiard flow  are dense in the whole phase space \cite{BoGaKrTr}.

On the other hand much less is known about the billiard in an irrational polygon.  The only global result
is that the billiard map (and thus flow) is always of zero entropy, and thus never $K$-mixing \cite{Ka1,GaKrTr,GuHa} and every non-periodic orbit
accumlates at a vertex \cite{GaKrTr}.
Scheglov has shown that for almost every triangle we have a stronger result, certain geometric quantities grow slower than a
stretched exponential \cite{Sc}.
Katok and Zemlyakov developed a
method to prove results about the dynamics of  typical polygons, here and throughout the article the word typical
means that the property holds for a dense $G_\delta$ set of polygons equipped by an appropriate topology. They developed this technique to  show that the billiard in a typical polygon is topologically transitive \cite{KaZe}.  Using this technique and the Kerchoff-Masur-Smillie result mentioned above, one can show
that he billiard in a typical polygon is  ergodic  with respect to the phase volume \cite{KeMaSm, Ka2, PoSt}.
Finally Troubetzkoy showed that the billiard map in a typical polygon is totally ergodic \cite{Tr}.

In this article we  use the approximation technique to give the first known result on higher mixing for  the
billiard in irrational polygons. We show

\begin{theorem}\label{t1}
For each $k \ge 3$ there
is a dense $G_\delta$ set $G$ of $k$-gons such that if $P \in G$ the  billiard map as well as the billiard shift are topologically weakly mixing.
\end{theorem}

All the notions in the theorem will be made precise in Section \ref{s3}.
 Besides the method of approximation the proof uses only two ingredients, the fact that the directional billiard map in most directions is
 minimal due to Keane \cite{Ke} from which one can easily deduce that it is totally transitive (totally minimal), and the fact that the billiard flow (and thus the billiard map) has a dense set of periodic orbits \cite{BoGaKrTr}.  Somewhat surprisingly, in contrast to the three previously mentioned results using the approximation technique, we do not know that the target property of topological weak mixing holds for directional billiard maps
 in rational polygons.

 The structure of the article is as follows, in Section \ref{s3} we will  state all the precise definitions and  give some background material necessary to prove Theorem
 \ref{t1}. Finally in Section \ref{s4} we will first state  a more general version of
 Theorem \ref{t1} and we prove this theorem.

\section{Definitions and background material}\label{s3}

\subsection{Topological dynamics}\label{pete}
A pair $(X, f)$, where $X$ is a compact Hausdorff topological space and $f : X \to X$ is
a homeomorphism is called a  \emph{topological dynamical system}.
A topological dynamical system $(X,f)$ is called \emph{topologically transitive} if the forward orbit $O(x) = \{f^ix: i \in \N\}$ of some point
$x \in X$ is dense in $X$. A necessary and sufficient condition for $(X,f)$ to be topologically transitive is that for each pair of open sets $U,V$
there exists an $n \in \N$ such that $f^n U \cap V \ne \emptyset$.
A topological dynamical system is called \emph{topologically weakly mixing} if the topological dynamical system $(X \times X, f \times f)$
is topologically transitive, i.e., the map $f \times f$ has a dense
orbit in $X \times X$.\footnote{There is also a different definition of topological weak mixing via continuous eigenfunction,  see \cite{KeRo}
for the relationship between these notions.}
The above mentioned characterization of topological transitivity shows that a topological dynamics system is  topologically weakly mixing if and only if for any four open sets $A,B,C,D \subset X$ there
exists a $n \in \N$ such that $(f \times f)^n(A \times B) \cap (C \times D) \ne \emptyset$, we will refer to this
characterization by (TM). Using this characterization
Petersen has shown that a topological dynamical system is  topologically weakly mixing if and only if given nonempty open sets $U$ and $V$ of $X$
there exists an $n \in \N$ such that $f^n U \cap U \ne \emptyset$ and $f^n U \cap V \ne \emptyset$ \cite{Pe}.
A topological dynamical system $(X,f)$ is called \emph{totally transitive} if $(X,f^n)$ is topologically transitive for all $n \ge 1$
and is called \emph{totally minimal} if $(X,f^n)$ is minimal for all $n \ge 1$.

\subsection{Definition of the billiard}
Details on polygonal billiards  can be found in the nice  survey article \cite{MaTa}.
Let $P \subset \R^2$ be a polygon.
A  point mass, moves inside  $P$
with unit speed along a straight line until it reaches the boundary
$\partial P$, then instantaneously changes direction according to the mirror
law:  ``the angle of incidence is equal to the angle of reflection,'' and
continues along the new line. If the trajectory hits a corner of the polygon,
in general it does not have a unique continuation and thus we decide that
it stops there.
The {\it phase space} of the \emph{billiard flow} $\phi_t$ is the quotient
\begin{multline*}
Y_P :=\big  \{ (x,\xi) \in P \times \S: \xi \text{ points into } P \text{ and }\\ x \text{ belongs to the interior of an edge } \big \} / \sim,
\end{multline*}
where $\sim$ identifies $(x,\xi)$ with  $(x,\delta \xi)$ whenever $x$ lies in the interior of an edge and $\delta$ is the reflection in that edge.

Let $C_P$ denote the \emph{set of corners} of $P$ and
$$X_P := Y_P \, \cap \, ((\dP \setminus C_P) \times \S).$$ The {\it billiard map} $f: X_P \To X_P$ is the first return of the flow $\phi_t$
to the set $X_P$ (i.e., the boundary $\dP$ of $P$);  for convenience we have removed the corners from the phase space.
We have $X_P = \cup_{i=1}^k    X_i$
where $X_i$ is an open rectangle whose  length is the length of the $i$th side of $P$, and whose  height is
$\pi$  since at each (non-corner) point there is a half circle of directions pointing inside $P$.
 We will denote points of $X_P$ by $u = (x,\theta)$ where the
inner
pointing directions $\theta\in\Theta = (-\frac{\pi}{2},\frac{\pi}{2})$ is measured with respect to the inner pointing normal.

\subsection{Polygonal topological dynamics}
The billiard flow and map are not continuous, they are not even defined everywhere. None the less
we extend the various  definitions from topological dynamics to the billiard map.
The billiard map
 is called \emph{topologically transitive} if there exists a point whose  forward orbit is dense.
The billiard map is called \emph{topologically weakly mixing} if the cartesian product of the billiard map
with itself is topologically transitive in the sense that it has a dense orbit.
We will use the characterizations of the previous section, but with care to make sure
that they hold in the polygonal setting.

\subsection{Interval exchange transformations}

Let $\beta_0 = 0 < \beta_1 \cdots < \beta_n = 1$.
An interval exchange transformation is a bijective map $f$ of the interval $[0,1)$ such that $f$ restricted to each of the intervals
$[\beta_i,\beta_{i+1})$ is a translation. The interval exchange $([0,1),f)$ has a \emph{saddle connection} if there is an forward orbit
segment
which starts at some $\beta_i$ and ends at some $\beta_j$ (possibly the same) or if there is a forward
 orbit segment of $\lim_{x \to \beta_i^-} fx$
which ends at some $\beta_j$.

We call the interval exchange $([0,1),f)$ \emph{minimal} if all forward orbits are dense.  We call it  \emph{totally minimal} if $([0,1),f^k)$ is minimal for all $k \ge 1$. There is a well known structure of interval exchange transformations due to Keane \cite{Ke}, who showed
that if
an interval exchange transformation
has no saddle connections then it is minimal. Consider an IET $([0,1),f)$ without saddle connections. The inverse map $([0,1),f^{-1})$
is also an IET without saddle connections.  Now consider the IETs
 $([0,1),f^k)$   where
$k  \in \Z \setminus \{0\}$.  Notice that if $f^k$  has a saddle connection, then $f$ has a saddle connection. Thus we can conclude.

\begin{lemma}\label{l2}
Any IET $([0,1),f)$ without saddle connections is totally minimal.
\end{lemma}
We remark that the inverse  $f^{-1}$ of an IET $f$ without saddle connections is also an IET without saddle connections
and thus is also totally minimal.

\subsection{Rational polygons}\label{ss1}
By $D$ we denote
the group generated by the reflections in the lines through the
origin, parallel to the sides of the polygon $P$. The group $D$ is
either

\begin{enumerate}[(i)]
\item{} finite, when  all the angles between sides of $P$ are of the form $\pi
m_i/n_i$ with distinct co-prime integers $m_i$, $n_i$,  in this
case $D=D_N$ the dihedral group generated by the reflections in
lines through the origin that meet at angles $\pi/N$, where $N := N_P$
is the least common multiple of $n_i$'s, or
\item{} countably infinite, when at least one angle
between sides of $P$ is an irrational multiple of $\pi$.
\end{enumerate}
In the two cases we will refer to the polygon as
{\it rational}, respectively {\it irrational}.

For $\xi\in \S \setminus \{ k\pi/N_P : k =0,1,\dots,2N_P-1\}$, let $\mathcal{S}_{\xi}$ be the set of points in $Y_P$ whose second coordinate belongs to the orbit of $\xi$ under $D$.
Since a trajectory changes its direction by an element of $D$ under
each reflection, $\mathcal{S}_{\xi}$ is an invariant set of the billiard
flow $\phi_t$ in $P$.
In fact if $P$ is a rational polygon then the set $\mathcal{S}_{\xi}$ is a compact surface
with conical singularities and a flat metric away from the singularities, i.e., a  translation
surface (\cite{MaTa} Section 1.5). Let $\mathcal{R}_\xi := \mathcal{S}_\xi \cap X_P$, we will also sometimes refer to $\mathcal{R}_\xi$ as $\mathcal{R}_{\theta}$
if  the angle is measured with respect to the normal. Throughout the article we will assume that are angles are not of the form $k\pi/N$ often without explicitly mentioning it.

A {\em saddle connection} is a  segment of a billiard orbit starting in a corner of $P$ and ending in
a corner of $P$.
 A direction is
{\it exceptional} if it is the direction of a saddle connection. There are countably many
saddle connections, hence the are countably many exceptional directions. A
direction which is not
exceptional will be called {\it non-exceptional}.

We recall  some well known results about rational polygonal billiards.
\begin{enumerate}[(i)]\addtocounter{enumi}{2}
\item{}  if $\theta$ is non-exceptional, then the directional billiard map on $\mathcal{R}_\theta$ can be extended by one sided-continuity to a minimal IET without  saddle connections (\cite{MaTa} p.\ 1027), and
\item{}  periodic points of the billiard flow are dense in the phase space \cite{BoGaKrTr}.
\end{enumerate}

Remark: the statement (iii) is a bit  awkward  since we have different conventions on the behavior at singular points, for IETs
the map is defined and continuous from the right at singular points, while for billiards the map is not defined when leaving or arriving
at a corner.
These two facts imply
\begin{corollary}\label{c1}
For any rational polygon
\begin{enumerate}[(i)]\addtocounter{enumi}{4}
\item{} if $\theta$ is non-exceptional, then the billiard map on $\mathcal{R}_{\theta}$ is totally transitive (even totally minimal), and\label{a}
\item{} periodic points of the billiard map are dense in the phase space.\label{c}
\end{enumerate}
\end{corollary}

\begin{proof}
Combining Lemma \ref{l2} and (iii) shows that in any non-exceptional direction,
any forward infinite orbit of any power of the  billiard map is dense in $\mathcal{R}_\theta$, proving (\ref{a}).

To see (\ref{c}) take a small open neighborhood $U$ in the phase space $X_P$ of the billiard map satisfying that the closure
$\overline{U}$  (in $\overline{X_P}$) does not contain a vertex or a vector tangent to a side.   Fix $\e > 0$ small and consider
 $U(\e) := \cup_{t \in (0,\e)} \phi_t(U)$.  For small enough $\e > 0$ this is an open neighborhood of the phase space of the billiard flow,
 thus there is a point $\bar{x} \in U(\e)$ which is periodic under the billiard flow.  There is a point $x \in U$ such that
 $\bar{x} = \phi_t x$ for some $t \in (0,\e)$.  So $x$ is a periodic point of the billiard map.
 \end{proof}

\subsection{Billiard shifts} Suppose that $P$ has $k$ sides, label each side of $P$ by a different letter from the alphabet
$\{1,2,\dots, k\}$.
Let $X_P'$ denote the set of  $u = (x,\theta) \in X_P$ for which $f^iu$ is defined for all $i \in \Z$.
The \emph{code of} $u$, denoted by $c(u)$, is  the sequence of labels of the sides visited by the orbit of $u$.
Let
$$\Sigma'_P: = \{c(u) : u \in X_P'\}.$$  $\Sigma'_P$  is  a subset of $ \{1,2\dots,k\}^{\Z}$ endowed with the product topology.
Let $\Sigma_P$ denote the closure of $\Sigma'_P$. A cylinder set is a set of the form
$\{z \in \Sigma_P: z_i = a_0, \dots, z_{i+k} = a_k\}$, cylinder sets form a basis of the topology.
We consider the left shift operator $\sigma$ on $\Sigma_P$.
We will refer to the topological dynamical system $(\Sigma_P,\sigma)$ as \emph{the billiard
shift}.

By definition  we have the following commutative relationship

\begin{equation} \begin{tikzcd}\label{e1}
X'_P \arrow{r}{f} \arrow[swap]{d}{c} & X'_P \arrow{d}{c} \\
\Sigma'_P \arrow{r}{\sigma}& \Sigma'_P
\end{tikzcd}
\end{equation}

Although we will not use it here, it seems appropriate to point out that
the map $c$  invertible up to periodic points.
 More precisely, for
each $c \in \Sigma_P'$ le $\pi(c)$ denote the set of $u \in X_P'$ such that the code $c(u)$ equals $u$.
Extend the set valued map $\pi$ to a set valued map from  $\Sigma_P \to \overline{X_P'} \subset \overline{X_P}$ by continuity (we will refer to this
extension by the same symbol $\pi$).
Galperin et. al.\ \cite{GaKrTr} have shown that

\begin{enumerate}[(i)]\addtocounter{enumi}{6}

  \item if $c  \in \Sigma_P$ is periodic with period $n$, then $\pi(c)$ is a horizontal interval in $\overline{X_P}$; when $n$ is even
  each point from $\pi(c)$ is a periodic point  period $n$,  while if $n$ is odd then the midpoint of the interval $\pi(c)$  has period $n$ while
  all  other points from
  $\pi(c)$ have period $2n$, and
\item if $c \in \Sigma_P$ is non-periodic then the set $\pi(c)$ consists of one point.
\end{enumerate}

\subsection{Topologies on the space of $k$-gons}

For simplicity suppose that the polygon is simply connected (our results hold in the finitely connected case
as well). To state our theorem, we need a topology on the set of all
$k$-gons (for $k \ge 3$ fixed). We will consider several such topologies, our results hold for all of the topologies.

The first topology is the simplest. A $k$-gon is determined by its corners, thus we can think of the $k$-gons as an open subset
of $\R^{2k}$. The topology of $\R^{2k}$ then induces a topology on the set of $k$-gons.
The space of all $k$-gons is thus a Baire space.
In this topology there is no attempt to represent a polygon in a unique way.

Notice that
the dynamics of the billiard in similar polygons is conjugate by the similarity.
We introduce second topology which represents the dynamical system of a billiard in a  polygon taking this into account.
Fix a corner of the polygon and enumerate the sides in a cyclic way starting at this corner (say in a clockwise fashion).
Place this corner at the origin and consider the similarity which takes the first side to the interval connecting the two points $(0,0)$ $(1,0)$.
Then the side-lengths
$\ell_1=1, \ell_2, \dots, \ell_k$ and the angles $\alpha_1, \dots, \alpha_k$  where $\alpha_i$ is the angle
 between sides $i$ and $i+1$ determine the polygon.
The angle $\alpha_1$ is determined by the other angles, and the lengths $\ell_2$ and $\ell_3$
are determined by the remaining angles and lengths.
Thus the set of $k$-gons are parametrized by $\alpha_2, \dots, \alpha_k, \ell_4,\dots,\ell_k$, i.e.,
 an open subset of $\R^{2k-4}$ (for triangles there are no length parameters).
 This representation is more unique than the first one.  Each billiard in a polygon is represented $2k$-times, namely there are
 two representations for each corner of the polygon (one for each orientation).
 Again the space of all $k$-gons is  a Baire space.

 Finally we can make the representation unique, but only for a subset of $k$-gons: those with a unique longest side, such that the
two sides of $P$ adjacent to this side have different lengths (call them \emph{generic}). In the above setting we
consider the similarity which takes this longest side to the interval connecting the two points $(0,0)$ $(1,0)$
such that the longer of the two adjacent sides starts at $(1,0)$ and points upwards, thus the set of generic $k$-gons is  an open subset of $\R^{2k-4}$. The non-generic $k$-gons are on the boundary of this set, but they do not have a unique representation.
Again the space of all generic $k$-gons is a Baire space.

\section{Proof of Theorem \ref{t1}}\label{s4}
Without any extra work we  will shown the following generalization of Theorem \ref{t1}.

\begin{theorem}\label{t4} Fix $k \ge 3$ and suppose that $Z$ is a Baire set of $k$-gons such that  for each $n \in \mathbb{N}$
the set $$\{P \in Z: P \text{ is a rational polygon such that } N_P \ge n\}$$  is dense in $Z$.
Then there is a dense $G_\delta$ subset $G$ of $Z$ for such that if $P \in G$ the corresponding billiard map and  billiard shift are  topologically weakly mixing.
\end{theorem}

The function $N_P$ was defined in Fact (i) of Subsection \ref{ss1}.
For example we could take $Z$ to be the set of right triangles. Theorem \ref{t1} follows immediately by taking $Z$ the set of
all $k$-gons.

We begin by a precision on the statement of the theorem,
it holds in both of the first two topologies, as well as for the third topology, but in this case
it is a statement (only) about generic polygons. There are no differences in the proof other than
the topology used.

Besides the method of approximation, the proof of the theorem uses some ideas  contained in the following theorem of Banks,
\emph{if $(X,f)$ is a totally transitive topological dynamical system and periodic points of $f$ are dense in $X$ then $(X,f)$ is topologically weakly mixing} (\cite{Ba} Theorem 1.1).

Fix a $k$-gon $P$ and the corresponding phase space $X_P$.
We (piecewise) affinely identify the each rectangle $X_i$ with the rectangle $S := (0,1) \times (0,\pi)$, i.e., we affinely normalize
the length of the $i$th side to be $1$.
For each $M \ge 1$ we consider the collection of  $1/M \times \pi/M$ rectangles contained in $S$ with
corners  at a point of the form  $(i/(2M),j\pi /(2M))$ for some integers $i,j$.
This collection is a finite open cover of $S$. It has cardinality $(2M-1)^2$. For each $M$ we denote this collection by $\U^M := \{U_i^M\}$. The affine identification then yields a finite open cover of $X_P$ of cardinality $k(2M-1)^2$.
The union of these covers
over $M \ge 1$ yields a countable basis $\U$ for the topology of $X_P$.
Since we are studying topological weak mixing we also define $\V^M := \{U_i \times U_j: U_i, U_j \in \U^M\}$. Then $\V := \cup_{M \ge 1} \V^M$ is a countable basis for the topology
of $X_P \times X_P$.
The phase spaces $X_P$  and $X_P \times X_P$ as well as these open covers
vary continuously as we vary $P$.

Now suppose that $P$ is rational. In this first step we use Fact (i).  Let $N = N_P$.
Then  for any $\xi$ not of the form $k\pi/N$, the invariant set $\mathcal{R}_\xi$ is at least $\pi/(2N)$ dense in the phase space $X_P$ since the $D_N$ orbit of
any $\theta$ is at least $\pi/(2N)$ dense in $\S$.
Thus we have
\begin{enumerate}[(i)]\addtocounter{enumi}{8}
\item if  $P$ is a rational polygon such that $N_P > M/2$, then for each $\xi$ not of the form $k\pi/N$  the invariant set $\mathcal{R}_\xi$ intersects all the sets  in $\U^M$.
\end{enumerate}

\begin{proof}[Proof of Theorem \ref{t4}]  
By our definition of the billiard map it is not defined when leaving or arriving at a corner.
For any set $A \subset X_P$ and any ${\ell} \in \mathbb{Z}$ let $A_{\ell}$ be the set of points in $A$ such that the map $f^{\ell}$ is defined.  
We will use the following short notation, we write $f^{\ell} A$ to stand for $f^{\ell} A_{\ell}$.
Note that the set of points in $X_P$ whose billiard flow orbit arrives at a corner  is a closed  smooth  curve in $X_P$,
thus the set of points for which $f$ is not defined is a finite union of closed smooth curves. The same is true for any iterate
$f^{\ell}$ of  of the billiard map (${\ell} \in \mathbb{Z} \setminus \{0\}$). We will repeatedly use the following implication of these facts: 
for any open set $U \subset X_P$ for each ${\ell} \in \mathbb{Z}$ the set  $f^{\ell}U$ is an open set, and furthermore if $U$ is non-empty
the $f^{\ell}U$ is non-empty. 
Furthermore, if $P$ is rational, and $U$ intersects a $\mathcal{R}_\xi$ then $U \cap \mathcal{R}_\xi$ is open in $\mathcal{R}_{\xi}$ and non-empty and thus for each ${\ell} \in \mathbb{Z}$ the set
$f^{{\ell}} U \cap \mathcal{R}_\xi$ is open in $\mathcal{R}_\xi$ and furthermore  non-empty since the set of points  in $\mathcal{R}_\xi$ for which the map $f^{\ell}$  is
not defined is a finite set of points.

Fix $k \ge 3$ and $M \ge 1$.
The first part of the proof resembles Petersen's characterization of weak mixing \cite{Pe} used in the proof of Banks theorem
\cite{Ba}.
Consider any rational $k$-gon $P$  in $Z$ such that $N_P = N > M/2$ and sets
$A,B,C,D \in \U^M$.  Consider any $\xi$ for which the map $(\mathcal{R}_{\xi},f)$ is (totally) transitive.
Using (ix), since $N > M/2$  each of the sets $A,B,C,D$ intersects $\mathcal{R}_\xi$ in an open set in $\mathcal{R}_\xi$,
thus there is a $\ell \ge 1$
such that  the open set   $E := f^\ell A \cap C$ is non-empty, and thus
$F := f^{-\ell} E \subset A \cap f^{-\ell} C$ is also open and non-empty (note that we do not necessarily have equality here since
$f^{-\ell} f^{\ell} E = E_\ell \subset E$, see the remark at the start of the proof).
By Corollary \ref{c1} periodic points are dense in $X_P$, thus there is a periodic point $x  \in F$.  Let $m$ be the period
of $x$.
For $j \ge 1$, one has $f^{jm +\ell}x = f^\ell x \in C$, and thus the open set $f^{jm+\ell} A \cap C$ is also non-empty.
The set $f^{-\ell} D$ is open and non-empty and intersects all $\mathcal{R}_\xi$. Thus we can choose a $\xi$ such that
 $(\mathcal{R}_\xi,f^{m})$ is transitive, since the open set $f^\ell B$ is non-empty this    implies that there is a $j \ge 1$ such that $f^{mj} (f^\ell B )\cap D \ne \emptyset$. Let $n = mj + \ell$,
 then we have shown that
 $(f\times f)^n (A \times B) \cap (C \times D)$ is a non-empty open set.

We have shown that  the billiard map in $P$ verifies  the condition for topological mixing for all sets $V,W \in \V^M$ (here $V := A \times B,  W := C \times D $).
Since the collection $\V^M$ is finite and the billiard map and the collection $\mathcal{V}^M$ vary continuously as we vary $P$
 we can find an $\e(P) > 0$ such that the billiard map verifies this condition for all $Q \in B(P,\e(P))$.

 Let $\mathcal{P}_N$ denote the set of rational polygons which satisfy $N_P = N$.
 Consider  $$\bigcup_{N > \frac{M}{2}} \bigcup_{P \in \mathcal{P}_N} B(P,\e(P)).$$
This is an open dense set of polygons verifying the characterization (TM)  for all
sets $V,W \in \V^M$.
Thus the set
$$G :=\bigcap_{M=1}^\infty \bigcup_{N > \frac{M}{2}} \bigcup_{P \in \mathcal{P}_N} B(P,\e(P))$$
is a dense $G_\delta$ set verifying the characterization (TM) for all $V,W$ in the basis $\mathcal{V}$.

Suppose $P \in G$. To see that
 that the billiard shift $(\Sigma_P,\sigma)$ is topologically weakly mixing
consider  a cylinder set $\A \subset \Sigma_P$.  Then since $\U$ is a basis for the topology of $X_P$
we can then find a (non-empty) $A \in \U$ such that  $A \subset \pi(\A) \cap X_P$.
Now repeat this for four cylinder sets $\A,\B,\D,\D$.  We have verified that there is an $n$ such that
$(f \times f)^n (A \times B) \cap (C \times D) \ne \emptyset$.  Then Equation  \eqref{e1} implies that
$(\sigma \times \sigma)^n (\A \times \B) \cap (\C \times \D) \ne \emptyset$. Thus by the characterization (TM)
the billiard shift $(\Sigma_P,\sigma)$ is topologically weakly mixing.

Now since the billiard map is not a topological dynamical system, we can not directly use the characterization (TM), but need to repeat
the ideas behind the proof of this characterization to show that it is   topological weakly mixing.

Consider $P \in G$ and $W,V \in \V$.
Since $P \in G$ there exists $n \in \N$ such that the open set $(f \times f)^n W \cap V$ is non-empty.
Thus we can choose
 an open set $W'$  whose closure $\overline{W'} \subset W$ such that $(f \times f)^{n}$ is defined (and thus continuous) for every point of
$\overline{W'}$ and  $(f \times f)^{n_1} (\overline{W'}) \subset V$.

Remember that  $\V$ is a countable basis for the topology of $X_P \times X_P$.
We need to find an infinite $f \times f$ orbit which visits each of the $V_{k}$. Fix  $W_0 \in \V$
such that $\overline{W_0} \subset X_P \times X_P$.
Apply the previous paragraph to  $W_0$ and $V_1$, this yields an $n_1 \in \N$ and a compact neighborhood
$\overline{W_1} \subset W_0$ such that $(f \times f)^{n_1}$ is defined and continuous on
$\overline{W_1}$ and $(f \times f)^{n_1} \overline{W_1} \subset V_1$.

Now inductively suppose that for $\ell \in \{1,\dots, m\}$ we have defined a compact neighborhoods $\overline{W_\ell} \subset W_{\ell -1}$, and positive integer $n_\ell$
such that for each $\ell \in \{ 1,2,\dots,m\}$ the map
$(f \times f)^{n_\ell}$ is defined and continuous on
$\overline{W_\ell}$ and $(f \times f)^{n_\ell} \overline{W_\ell} \subset V_\ell$.

Now we again  apply the above construction to $W_m$ and $V_{m+1}$ to find
 an $n_{m+1} \in \N$ and a compact neighborhood
$\overline{W_{m+1}} \subset W_m$ such that $(f \times f)^{n_{m+1}}$ is defined and continuous on
$\overline{W_{m+1}}$ and $(f \times f)^{n_{m+1}} \overline{W_{m+1}} \subset V_{m+1}$.

Every point in $\cap_{m \ge 0} \overline{W_m}$ will have an  infinite $(f \times f)$ forward orbit which visits all the sets $V_\ell$, and
thus the forward orbit is dense in $X_P \times X_P$.
This finishes the proof of Theorem \ref{t4} and thus also of Theorem \ref{t1}.
\end{proof}

\section{Open problems}

It would be interesting to prove that the billiard flow in a typical polygon is topologically weakly mixing. One could
envisage a proof similar to the one we have given for the billiard map, but one would need to know something about the transitivity of time $t$ map for a dense
set of invariant surfaces in a rational polygon for all times $t$ which are periods of the billiard flow.

The aforementioned result of Chaika and Fickensher indicates that it could be possible that the billiard map and or flow is
topologcally mixing in some irrational polygons.

The measure theoretic mixing properties stronger than ergodicity and weaker than $K$-mixing (for example weak mixing, strong mixing, multiple mixing) for the billiard map and or flow in any irrational polygon is a completely open question.

\end{document}